\documentclass[12pt,reqno]{amsart}
\usepackage{amssymb,amsmath,comment}
\usepackage{color}
\def\({\left(}
\def\){\right)}
\def\Nx{\nabla_x}

\def\eb{\varepsilon}

\def\R {\mathbb{R}}

\newcommand{\be}{\begin{equation} }
\newcommand{\ee}{\end{equation} }

\def \and{\qquad\text{and}\qquad}

\def\Dt{\partial_t}

\def\({\left(}
\def\){\right)}
\def\Nx{\nabla}

\def\eb{\varepsilon}

\def\eb{\varepsilon}

\def\R {\mathbb{R}}

\def \and{\qquad\text{and}\qquad}

\def\Dt{\partial_t}

\newtheorem{proposition}{Proposition}[section]
\newtheorem{theorem}[proposition]{Theorem}
\newtheorem{corollary}[proposition]{Corollary}

\theoremstyle{definition}

\newtheorem{remark}[proposition]{Remark}

\numberwithin{equation}{section}

\def\be{\begin{equation}}
\def\ee{\end{equation}}
\def\bp{\begin{proof}}
\def\ep{\end{proof}}

\def \no#1#2#3 {{\bf #1} (#3), #2.}
\def \eds#1#2#3 {#1, #2, #3.}

\title[Preventing blow up by convection]
{Preventing blow up by convective terms in dissipative PDEs}

\author[B. Bilgen, V. Kalantarov and S. Zelik] {Bilgesu Bilgen ${}^1$, Varga Kalantarov ${}^2$ and Sergey Zelik${}^2$}
\address{${}^3$
University of Surrey, Department of Mathematics,
Guildford, GU2 7XH, United Kingdom, s.zelik@surrey.ac.uk.}
\address{${}^{2,1}$ Department of mathematics, Ko{\c c} University, Rumelifeneri Yolu, Sariyer, Istanbul, Turkey, vkalantarov@ku.edu.tr, bilbilgin@ku.edu.tr}

\begin{document}

\begin{abstract}
We study the impact of the convective terms on the global solvability or finite time blow up of solutions of dissipative PDEs. We consider the model examples of 1D Burger's type equations, convective Cahn-Hilliard equation, generalized Kuramoto-Sivashinsky equation and KdV type equations, we establish the following common scenario: adding sufficiently strong (in comparison with the destabilizing nonlinearity) convective terms to equation prevents the solutions from blowing up in finite time and makes the considered system globally well-posed and dissipative and for weak enough convective terms the finite time blow up may occur similarly to the case when the equation does not involve convective term.
\par
This kind of result has been previously known  for the case of Burger's type equations  and has been strongly based on maximum principle. In contrast to this, our results are based on the weighted energy estimates which do not require the maximum principle for the considered problem.
\end{abstract}
\subjclass[2010]{35A01, 35B44, 35B45}
\keywords{Global existence, preventing blow-up, Burger's type equations,
Kuramoto-Sivashynsky equation, convective Cahn-Hilliard equation, KdV type equations}
\thanks{This work is partially supported by  the grant  14-41-00044 of RSF}
\maketitle
\tableofcontents

\section{Introduction}
It is well-known that the solutions of nonlinear evolutionary PDEs may blow up in a finite time. The most studied is the case of a semilinear heat equation, for instance, all positive solutions of the problem
\begin{equation}\label{0.heat}
\partial_t u=\partial^2_x u+u^2,\ \ u\big|_{x=\pm1}=0,\ \ u\big|_{t=0}=u_0
\end{equation}
blow up in finite time, see e.g., \cite{QS} and references therein, see also \cite{AlKoSv,Ba,BeTi,Gal, KaLa,Le1,Le2,LePa,Po} for the analogous results for more complicated equations.
\par
It is also known  that the presence of convective terms may prevent blow up and makes the problem globally well-posed. In particular, as shown in \cite{ChLeSa} and \cite{LePaSaSt},
the solutions of the problem
\begin{equation}\label{0.heat-bur}
\partial_t u+\varepsilon u\partial_xu =\partial^2_x u+u^2,\ \ u\big|_{x=\pm1}=0,\ \ u\big|_{t=0}=u_0
\end{equation}
which differs from \eqref{0.heat} by the presence of the convective term $\eb u\partial_xu$ becomes globally well-posed if  $\eb\ne0$. However, the proof of this fact given there is strongly based on the maximum principle and the situation becomes  essentially less clear for more complicated equations where the impact of the convective terms is a bit controversial. Indeed, for example, in the case of Navier-Stokes equations, the convective/inertial form $(u,\nabla_x)u$ is the only nonlinearity in the system and the only source of instability and possible blow up in the 3D case. On the other hand, if we consider the associated vorticity equation
\begin{equation}\label{0.vort}
\partial_t\omega+(u,\nabla_x)\omega=\Delta_x \omega+(\omega,\nabla_x)u
\end{equation}
and replace (very roughly) the term $\nabla_x u$ by $\omega$ in the vorticity stretching term and $u$ by $\omega$ in the vorticity transport term, we end up with the model equations similar to \eqref{0.heat}. This analogy can be made more precise, namely, as shown in \cite{fursikov}, the quadratic nonlinearity $B(\omega):=(\omega,\nabla_x)u-(u,\nabla_x)\omega$ can be split into the sum of two non-local operators $B(\omega)=B_\tau(\omega)+B_n(\omega)$ such that the tangential component $B_\tau(\omega)$ satifies
$$
(B_\tau(\omega),\omega)=0
$$
and can be interpreted as a generalized convective term and the equation
\begin{equation}\label{0.vort-norm}
\partial_t\omega=\Delta_x \omega +B_n(\omega),
\end{equation}
where only the normal component of the nonlinearity is presented possesses solutions blowing up in finite time. Thus, the Millennium problem concerning the Navier-Stokes equations can be reduced to the question whether or not the "convective" term $B_\tau(\omega)$ is strong enough to prevent blow up in the normal parabolic system \eqref{0.vort-norm}.
\par
Unfortunately, the above question seems out of reach of the modern methods, so the main aim of the present paper is to analyze the impact of the convective terms  for simpler equations where more or less complete answer is available despite the absence of the maximum principle. As the model examples, we consider the generalized Kuramoto-Sivashinsky equation, the so-called convective Cahn-Hilliard equation and the Korteveg de Vries type equations on a finite interval with Dirichlet boundary conditions. As we will see below, despite the fact that these equations have essentially different structure, they can be treated in the unified way and the obtained results look very similar to the case of Burgers type equations studied before. Namely, if the convective term is strong enough (in comparison with the destabilizing non-linearity), it prevents blow up of solutions and makes the considered problem globally well-posed and dissipative and if it is not sufficiently strong, the blow up may occur  in the convective equations as well. We believe that this picture has a general nature and may be useful in the study of more complicated cases including  the 3D Navier-Stokes equations.
\par
The paper is organized as follows.
\par
In Section \ref{s1} we revisit the case of 1D Burger's type equations:
\begin{equation}\label{0.eq1}
\partial_tu+\partial_x(|u|^{p+1})=\partial_x^2u+f(u)+g,\ \ u\big|_{x=\pm1}=0,\ \ u\big|_{t=0}=u_0,
\end{equation}
where the exponent $p>0$, the external force $g\in L^2(-1,1)$ and the destabilizing nonlinearity is growing not faster than $|u|^{q+1}$ for some other exponent $q>0$, namely, $f\in C^1(\R)$ and
\begin{equation}\label{0.f}
|f'(u)|\le C(1+|u|^q).
\end{equation}
The typical nonlinearities have the form $f(u)=u|u|^q$ and $f(u)=|u|^q$ although other choices are also allowed. Note also that the convective term of the form $\partial_x(u|u|^p)$ is also allowed here and can be treated even in a simpler way.
\par
As we have already mentioned, equation \eqref{0.eq1} has been studied in \cite{LePaSaSt} and \cite{ChLeSa}, see also \cite{SoWe1,SoWe2,Te1,Te2} and more or less complete answer on the blow up/global solvability question has been obtained: blow up is prevented if $p\ge q$ and occurs at least for some solutions of this equation if $p<q$ (e.g., for the non-linearity $f(u)=|u|^{q+1}$).
\par
In the present paper, we present an alternative/simplified method of proving this result which based on the weighted energy estimates (multiplication of the equation on $ue^{-Lx}$ or $u_+e^{-Lx}-u_-e^{Lx}$ where $L$ is a properly chosen parameter) rather than maximum principle and which will be used throughout of the paper. In addition, in Section \ref{s1}, we apply this method to some model 2D Burger's type system without the maximum principle.
\par
In section \ref{s2}, we study the 1D forth order parabolic equations with convective terms. The first considered example is the generalized Kuramoto-Sivashinsky equation:
\begin{equation}\label{0.eq2}
\partial_t u+\partial_x^4 u-\lambda\partial^2_x u+\partial_x(u|u|^p)=f(u)+g,\ u\big|_{x=\pm1}=\partial_x^2u\big|_{x=\pm1}=0,\ u\big|_{t=0}=u_0,
\end{equation}
where $\lambda\in\R$ and the parameters $p,q$, the external force $g$ and the nonlinearity $f$ are the same as in the case of Burger's equation. The proved result for this equation is also almost the same as for the Burger's equation: the blow up is prevented if $p\ge q$ and remains possible if $p<q$. The only difference is that since we do not have not $L^s$-estimates for $s\ne2$, we have to require extra restriction $p\le6$ to obtain the global well-posedness and regularity in the case $p\ge q$ (actually we do not know whether or not this extra restriction is essential).
\par
The second considered example is the convective Cahn-Hilliard equation:
\begin{equation}\label{0.eq3}
\partial_t u+\partial_x^2(\partial^2_x u+f(u))+\partial_x(u|u|^p)=f(u)+g,\ u\big|_{x=\pm1}=\partial_x^2u\big|_{x=\pm1}=0,\ u\big|_{t=0}=u_0,
\end{equation}
where the parameters $p,q$, the external force $g$ and the nonlinearity $f$ are the same as in the case of Burger's equation. This equation looks more complicated since the destabilizing nonlinearity $\partial_x^2f(u)$ is stronger due to the presence of the second derivative. The particular case of this equation with the convective term $\partial_x(u^2)$ and periodic boundary conditions has been studied in our previous paper \cite{EdKaZe}, see also \cite{EdKa} and references therein. In particular, the following partial results on global solvability/blow up have been obtained there: blow up is prevented if $q<\frac49$ and is occurred at least for some solutions if $q\ge2$. So, the behavior of solutions remains unclear if $\frac49\le q<2$. The results proved in this paper using the different method of weighted estimates are slightly better, but still not optimal, namely, the global well-posedness is verified if $q\le p/2$ (and $p\le6$) and the blow up is proved if $q>p+1$. Thus, the behavior of solutions is not clear if $p/2< q\le p+1$.
\par
Finally, in Section \ref{s3}, we study the 3rd order KdV type equation:
\begin{equation}\label{0.eq4}
\partial_t u+\partial_x^3u=\partial_x(u|u|^p)+f(u)+g,\ \ u\big|_{x=\pm1}=\partial_xu\big|_{x=1}=0,\ u\big|_{t=0}=u_0,
\end{equation}
where the parameters $p,q$, the external force $g$ and the nonlinearity $f$ are the same as in the case of Burger's equation. It worth emphasizing that, in contrast to the case of KdV equations on the whole line or with periodic boundary conditions where the corresponding equations are conservative and even completely integrable, they are {\it dissipative} if the Dirichlet boundary conditions are posed. Although the corresponding problem is not parabolic, it possesses the smoothing property on a finite time interval and the linear part generates a $C^\infty$-semigroup, see e.g., \cite{KdV} and references therein. Moreover, the analytic properties of this equation is similar to Burger's type equations (of course, up to the maximum principle). This essentially simplifies the proof of local well-posedness in comparison to the conservative case. As we show, similar to the Burger's type equations, the blow up is prevented if $p\ge q$ (under the extra restriction $p\le2$ which is posed in order to prove the global well-posedness in the phase space $L^2(-1,1)$) and in the case $p<q$ blow up remains possible and is occurred at least for some solutions in the case where $f(u)=|u|^{q+1}$.

\section{Burger's type equations}\label{s1}
The aim of this section is to illustrate the impact of convective terms to reaction-diffusion equations and systems. We start with the simplest model example of 1D Burgers equation:
\begin{equation}\label{1.eq1}
\partial_t u+\partial_x(u^2)=\partial_x^2 u+k u^2+g,\ \  u\big|_{x=\pm1}=0, \ \  u\big|_{t=0}=u_0,
\end{equation}
considered on the interval $x\in(-1,1)$ and endowed by the Dirichlet boundary conditions.
It is well-known that without the convective term $\partial_x(u^2)$ this equation possesses solutions blowing up in finite time if $k\ne0$. Moreover, as it is shown in \cite{LePaSaSt} (see also references therein), the presence of the convective term prevents the solutions to blow up and the global existence of solutions holds for any $k\in\mathbb R$. We present below an alternative/simplified proof of this fact which can be extended to much wider class of equations. Namely, the following result holds.
\begin{theorem}\label{Th1.bur} Let $k\in\R$ and $g\in L^2(-1,1)$. Then, for any $u_0\in L^2(-1,1)$, equation \eqref{1.eq1} possesses a unique solution $u(t)$ defined globally in time $t\ge0$ and this solution satisfies the following dissipative estimate:
\begin{equation}\label{1.estdiss-bur}
\|u(t)\|_{L^2}^2\le C\|u_0\|_{L^2}^2e^{-\alpha t}+C(\|g\|^2_{L^2}+1),
\end{equation}
 where the positive constants $C$ and $\alpha$ are independent of $u_0$ and $t$.
 \end{theorem}
 \begin{proof} The local well-posedness of problem \eqref{1.eq1} is straightforward, so we concentrate only on the derivation of dissipative estimate \eqref{1.estdiss-bur}. To this
end, we multiply equation \eqref{1.eq1} by by $u_+(t)e^{-Lx}-u_-(t)e^{Lx}$ where $u_+=\max\{u,0\}$, $u_-=u-u_+$ and $L$ is sufficiently large positive constant which will be specified below. Then, after the standard transformations, we end up with
\begin{multline}\label{1.huge}
\frac12\frac d{dt}\(\|u_+\|^2_{L^2_{e^{-Lx}}}+\|u_-\|^2_{L^2_{e^{Lx}}}\)+\|\Nx u_+\|^2_{L^2_{e^{-Lx}}}+\|\Nx u_-\|^2_{L^2_{e^{Lx}}}=\\=
\frac12L^2\(\|u_+\|^2_{L^2_{e^{-Lx}}}+\|u_-\|^2_{L^2_{e^{Lx}}}\)+\\+(k- L)\(\|u_+\|^{3}_{L^{3}_{e^{-Lx}}}+\|u_-\|^{3}_{L^{3}_{e^{Lx}}}\)+(g,u_+(t)e^{-Lx}-u_-(t)e^{Lx}),
\end{multline}
where we denote by $L^p_\varphi$ the weighted Lebesgue space $L^p$ with the norm
$$
\|u\|^p_{L^p_\varphi}:=\int_{-1}^1|u(x)|^p\varphi(x)\,dx.
$$
Crucial for our purposes is the fact that the weighted norms $\|u\|_{L^p_{e^{Lx}}}$ are equivalent to the non-weighted ones when the underlying domain is bounded.
\par
Fixing $L:=k+1$ in \eqref{1.huge} and using the Young inequality, we have
\begin{equation}
\frac12\frac d{dt}\(\|u_+\|^2_{L^2_{e^{-Lx}}}+\|u_-\|^2_{L^2_{e^{Lx}}}\)+
\(\|u_+\|^2_{L^2_{e^{-Lx}}}+\|u_-\|^2_{L^2_{e^{Lx}}}\)\le C(\|g\|^2_{L^2}+1).
\end{equation}
Applying the Gronwall inequality to this relation, we end up with the desired estimate \eqref{1.estdiss-bur} and finish the proof of the theorem.
 \end{proof}
The above result can be easily generalized to the  problem
\begin{equation}\label{1.eq2}
\partial_t u+\partial_x(|u|^{p+1})=\partial_x^2 u+f(u)+g,\ \ u\big|_{t=0}=u_0,
\end{equation}
where the nonlinear function $f$ satisfies \eqref{0.f} with the restriction $q\le p$.
Indeed, repeating word by word the derivation of estimate \eqref{1.estdiss-bur}, we end up with the same $L^2$-dissipative estimate for the solutions of equation \eqref{1.eq2}. However, when the exponent $q$ is large enough, the only dissipative estimate in $L^2$ is not sufficient to prevent blow up in higher norms and verify the global existence and dissipativity of {\it smooth} solutions. To this end, we need the dissipative control of $L^s$-norm
for some $s\ge s(p)<\infty$ (actually, it is not difficult to show that $s(p)=p$). Namely, the following result holds.
\begin{theorem}\label{Th1.bur1} Let $q\le p$, $p>0$ and  $u(t)$ be a sufficiently regular solution of \eqref{1.eq2}. Then,  for every $s\ge2$, the following estimate holds:
\begin{equation}\label{1.est-ls}
\|u(t)\|_{L^s}^s+\int_{t}^{t+1}\|\partial_x(u^{s/2}(\tau))\|^2_{L^2}\,d\tau\le C_s\|u_0\|^s_{L^s}e^{-\alpha t}+C_s(\|g\|_{L^2}^s+1),
\end{equation}
where $\alpha>0$ and the constant $C_s$ depends on $s$, but is independent of $u_0$ and $t$.
\end{theorem}
\begin{proof} We multiply equation \eqref{1.eq2} by $u_+(t)^{s-1}e^{-Lx}-u_-(t)^{s-1}e^{Lx}$ and integrate over  the interval $(-1,1)$. Then,  similarly  to \eqref{1.huge}, but using \eqref{0.f} in order to estimate the term with the non-linearity $f$, we have
\begin{multline}\label{1.huge1}
\frac1s\frac d{dt}\(\|u_+\|^s_{L^s_{e^{-Lx}}}+\|u_-\|^s_{L^s_{e^{Lx}}}\)+
\frac{4(s-1)}{s^2}\((|\partial_x(u_+^{s/2})|^2,e^{-Lx})+(|\partial_x(u_-^{s/2})|^2,e^{Lx})\)\le\\\le
\frac{L^2}{s}\(\|u_+\|^{s}_{L^{s}_{e^{-Lx}}}+\|u_-\|^{s}_{L^s_{e^{Lx}}}\)+
(g,u_+^{s-1}e^{-Lx}-u_-^{s-1}e^{Lx})+\\+
C\(\|u_+\|^{s+q}_{L^{s+q}_{e^{-Lx}}}+\|u_-\|^{s+q-}_{L^{s+q}_{e^{Lx}}}+1\)-
L\(\|u_+\|^{s+p}_{L^{s+p}_{e^{-Lx}}}+\|u_-\|^{s+p}_{L^{s+p}_{e^{Lx}}}\).
\end{multline}
Using the embedding $H^1\subset L^\infty$ (for the 1D case), we estimate the term containing the external forces $g$ as follows:
\begin{multline}\label{1.ext}
|(g,u_+^{s-1}e^{-Lx}-u_-^{s-1}e^{Lx})|\le C_L\|g\|_{L^2}\|u\|_{L^\infty}^{s-1}\le\\\le C_L\|g\|^s_{L^2}+
\frac{2(s-1)}{s^2}\((|\partial_x(u_+^{s/2})|^2,e^{-Lx})+(|\partial_x(u_-^{s/2})|^2,e^{Lx})\).
\end{multline}
Using now the Young inequality and together with the facts that $p\ge q$ and $p>1$ and fixing $L>C$, we end up with the differential inequality
\begin{multline*}
\frac d{dt}\(\|u_+\|^s_{L^s_{e^{-Lx}}}+\|u_-\|^s_{L^s_{e^{Lx}}}\)+
\frac{2(s-1)}{s}\((|\partial_x(u_+^{s/2})|^2,e^{-Lx})+(|\partial_x(u_-^{s/2})|^2,e^{Lx})\)+\\+\alpha \(\|u_+\|^s_{L^s_{e^{-Lx}}}+\|u_-\|^s_{L^s_{e^{Lx}}}\)\le C_s(\|g\|^s_{L^s}+1).
\end{multline*}
The Gronwall inequality applied to this relation gives the desired dissipative estimate and finishes the proof of the theorem.
\end{proof}
\begin{corollary}\label{Cor1.well} Under the assumptions of Theorem \ref{Th1.bur1} problem \eqref{1.eq2} is globally well-posed and dissipative in, say, $H^1_0(-1,1)$.
\end{corollary}
Indeed, the local well-posedness of this problem in $H^1_0$ is straightforward and the global solvability follows in a standard way from the dissipative estimate \eqref{1.est-ls} with $s\ge p-1$ and the parabolic smoothing property, see e.g., \cite{LaSoUr}, so we left details to the reader.
\par
Let us now consider the case $q>p$. In this case, as  it is shown in the next theorem, the convective term is not  strong enough to prevent finite time blow up  of solutions.

\begin{theorem}\label{Th1.blowup} Let $q>p$, $q>1$ and $k\ne0$. Then, there exists a smooth initial data $u_0$ such that the corresponding solution $u(t)$ of equation \eqref{1.eq2} with $f(u)=|u|^{q+1}$ blows up in finite time.
\end{theorem}
\begin{proof} Without loss of generality we may assume that $k>0$ (otherwise, it is enough to replace $u$ by $-u$. Let us introduce a weight function $\varphi(x)=(\eb^2-x^2)^n$, where $\eb>0$ and $n>0$ is a sufficiently large number which will be specified below, multiply equation \eqref{1.eq2} by $\varphi$ and integrate over $x\in(-\eb,\eb)$. Then, after the direct calculations, we get
\begin{equation}\label{1.dif}
\frac d{dt}\int_{-\eb}^\eb u\varphi\,dx=\int_{-\eb}^\eb|u|^{q+1}\varphi\,dx-\int_{-\eb}^\eb|u|^{p+1}\varphi'\,dx+\int_{-\eb}^\eb u\varphi''\,dx+\int_{-\eb}^\eb g\varphi\,dx.
\end{equation}
Since $q>p$, by the H\"older inequality, we have
$$
\Big|\int_{-\eb}^\eb|u|^{p+1}\varphi'\,dx\Big|\le\frac 14\int_{-\eb}^\eb|u|^{q+1}\varphi\,dx+
C\int_{-\eb}^\eb|\varphi'|^{\frac {q+1}{q-p}}\varphi^{\frac{-1-p}{q-p}}\,dx.
$$
Note that, for all $n\ge \frac{q+1}{q-p}$, the second integrand in the right hand side has no singularities at $x=\pm\eb$ and can be majorated by a constant. Analogously,
\begin{equation}\label{1.2d}
\Big|\int_{-\eb}^\eb u\varphi''\,dx\Big|\le \frac 14\int_{-\eb}^\eb|u|^{q+1}\varphi\,dx+C\int_{-\eb}^\eb|\varphi''|^{\frac {q+1}{q}}\varphi^{\frac{-1}{q}}\,dx
\end{equation}
and the last integral in the right-hand side can be majorated by a constant if $n\ge\frac{2(q+1)}{q}$. Finally, the Jensen inequality gives
$$
\Big|\int_{-\eb}^\eb u\varphi\,dx\Big|^{q+1}\le C\int_{-\eb}^\eb|u|^{q+1}\varphi\,dx
$$
and inserting the obtained estimates into the right-hand side of \eqref{1.dif}, we get
$$
\frac d{dt}\(\int_{-\eb}^\eb u\varphi\,dx\)\ge \alpha\Big|\int_{-\eb}^\eb u\varphi\,dx\Big|^{q+1}-C
$$
for some positive $\alpha$ and $C$ which are independent of $u$. The last inequality shows that the solution $u(t)$ indeed blows up in a finite time if
$$
\int_{-\eb}^\eb u_0(x)\varphi(x)\,dx>\frac {C^{1/(q+1)}}{\alpha^{1/(q+1)}}
$$
and the theorem is proved.
\end{proof}
\begin{remark}\label{Rem1.bound} Note that the results on preventing blow up by convective terms, see theorems \ref{Th1.bur} and \ref{Th1.bur1} strongly depend on boundary conditions. Indeed, if we take, say, Neumann boundary conditions for equation \eqref{1.eq1}, the effect will disappear since the spatially homogeneous solutions will blow up no matter how strong the convective term is. On the other hand, the result of Theorem \ref{Th1.blowup} is based on the interior estimates and will hold no matter what the boundary conditions are.
\par
Mention also that the proof of theorems \ref{Th1.bur} and \ref{Th1.bur1} becomes {\it essentially} simpler if the convective term has the form $\partial_{x}(u|u|^{p})$. Indeed, in this case, we need not to use the functions $u_+$ and $u_-$ and may multiply the equation by $u(t)e^{-Lx}$. Then only the (weighted) energetic arguments are used and this allows us to apply the technique to more general equations (e.g., fourth or third order equations where the maximum principle does not hold). We consider these cases in the next sections.
\end{remark}
We conclude this section by considering the model 2D Burgers type equations where the maximum principle does not work any more, but the above arguments will allow us to establish the blow up preventing by convection. Namely, let us consider a system
\begin{equation}\label{1.eq3}
\partial_t u+\operatorname{div}(u\otimes u)=\Delta_x u+f(u)+g,\ \ u\big|_{\partial\Omega}=0
\end{equation}
where $\Omega$ is a bounded smooth domain of $\R^2$ and $u=(u_1,u_2)$ is an unknown vector field.
The convective term here
\begin{equation}
\operatorname{div}(u\otimes u):=\(\begin{matrix} \partial_{x_1}(u_1^2) &\partial_{x_2}(u_1u_2)\\
                                                \partial_{x_1}(u_1u_2) &\partial_{x_2}(u_2^2)\end{matrix}\)
\end{equation}
is the standard convective term for the Navier-Stokes system and it is clearly non-monotone, so the methods based on the maximum principle will not work at least directly.
\par
We also assume that $g\in L^2(\Omega)$ and the non-linearity $f\in C^1(\R^2,\R^2)$ has the form
\begin{equation}\label{f.cond}
f(u)=\(\begin{matrix} k_1u_1^2\\k_2u_2^2\end{matrix}\)+\bar f(u),
\end{equation}
where $|\bar f'(u)|\le C(1+|u|^{1-\eb})$ for some $\eb>0$ and $k_1,k_2\in\R$. The next theorem is the analogue of Theorem \ref{Th1.bur} for this problem.
\begin{theorem}\label{Th1.bur3} Let the nonlinearity $f$ satisfy \eqref{f.cond} and let $u(t)$ be a sufficiently regular solution of equation \eqref{1.eq3}. Then, the following estimate holds:
\begin{equation}\label{1.est-l1}
\|u(t)\|_{L^1}+\int_{t}^{t+1}\|u(s)\|^2_{L^2}\,ds\le C\|u_0\|_{L^1}e^{-\alpha t}+C(\|g\|_{L^2}+1),
\end{equation}
where the positive constants $C$ and $\alpha$ are independent of $u_0$ and $t$.
\end{theorem}
\begin{proof} We multiply the first and the second equation of \eqref{1.eq3} by $\operatorname{sgn}((u_1)_+)e^{-Lx_1}-\operatorname{sgn}((u_1)_+)e^{Lx_1}$ and $\operatorname{sgn}((u_2)_+)e^{-Lx_2}-\operatorname{sgn}((u_2)_+)e^{Lx_2}$ respectively, take a sum and  integrate over $x\in\Omega$.  Then, thanks to  the Kato inequality (see, e.g., \cite{HiSi}) we have
$$
\int_{\Omega} \Delta u(x)\operatorname{sgn} u_+(x)\,dx\le 0.
$$
Employing this inequality, after  the straightforward transformations, we end up with
\begin{multline}
\frac d{dt}\(\|(u_1)_+\|_{L^1_{e^{-Lx_1}}}+\|(u_1)_-\|_{L^1_{e^{Lx_1}}}+
\|(u_2)_+\|_{L^1_{e^{-Lx_2}}}+\|(u_2)_-\|_{L^1_{e^{Lx_2}}}\)+\\+
(L-|k_1|-|k_2|)\(\|(u_1)_+\|_{L^2_{e^{-Lx_1}}}^2+\|(u_1)_-\|_{L^2_{e^{Lx_1}}}^2+
\|(u_2)_+\|_{L^2_{e^{-Lx_2}}}^2+\|(u_2)_-\|_{L^2_{e^{Lx_2}}}^2\)\le\\\le
L^2\(\|(u_1)_+\|_{L^1_{e^{-Lx_1}}}+\|(u_1)_-\|_{L^1_{e^{Lx_1}}}+
\|(u_2)_+\|_{L^1_{e^{-Lx_2}}}+\|(u_2)_-\|_{L^1_{e^{Lx_2}}}\)+\\+C(1+\|g\|_{L^2}+\|u\|_{L^{2-\eb}}^{2-\eb}).
\end{multline}
Fixing $L>|k_1|+|k_2|$ and applying Young and Gronwall inequality to this relation, we end up with the desired estimate \eqref{1.est-l1} and finish the proof of the theorem.
\end{proof}
\begin{remark}\label{Rem1.blowup} It is not difficult to see that without the convective term, the solutions of \eqref{1.eq3} may blow up in finite time. Also, if we take the nonlinearity $f$ growing faster than quadratically, we may construct the blow up solutions arguing as in Theorem \ref{Th1.blowup}.
\end{remark}
Note, however, that the dissipative estimate \eqref{1.est-l1} is too weak in order to verify that the higher norms of solutions do not blow up in finite time. To overcome this  obstacle, we need either to obtain stronger dissipative estimate (which we actually do not know how to do at the moment) or somehow regularize equations \eqref{1.eq3}. For instance, it can be done in the spirit of the $\alpha$-models for the Navier-Stokes equations (see \cite{FHT}), namely, let us consider the following regularized system:
\begin{equation}\label{1.eq4}
\begin{cases}
\partial_t u_1+\partial_{x_1}(u_1^2)+\partial_{x_2}(u_1v_2)=\Delta u_1+k_1 u_1^2+\bar f_1(u)+g_1,\\
\partial_t u_2+\partial_{x_2}(u_2^2)+\partial_{x_1}(u_2v_1)=\Delta u_2+k_2 u_2^2+\bar f_2(u)+g_2,\\
v=(1-\alpha\Delta)^{-1}u,\ \ u,v\big|_{\partial\Omega}=0,\ \ u\big|_{t=0}=u_0,
\end{cases}
\end{equation}
where $\alpha>0$ is a regularization parameter. Then, on the one hand, repeating word by word the proof of Theorem \ref{Th1.bur3}, we see that the solution \eqref{1.eq4} satisfies the dissipative estimate \eqref{1.est-l1} uniformly with respect to $\alpha$. On the other hand, the dissipative estimates in higher norms can be obtained from \eqref{1.est-l1} using the standard bootstrapping arguments. Indeed, since the key convective nonlinearity in \eqref{1.eq4} is similar to the Navier-Stokes one, it is enough to deduce the dissipative $L^2$ estimate which is done in the following corollary.
\begin{corollary}\label{Cor1.bur-l2} Under the above assumptions, problem \eqref{1.eq4} is globally well-posed in $L^2(\Omega)$ and the following estimate holds:
\begin{equation}\label{1.est-bur4}
\|u(t)\|^2_{L^2}+\int_t^{t+1}\|\nabla_x u(s)\|^2_{L^2}\,ds\le Q(\|u_0\|_{L^2})e^{-\gamma t}+Q(\|g\|_{L^2}),
\end{equation}
where $\gamma>0$ and the monotone function $Q$ is independent of $u_0$ and $t$.
\end{corollary}
\begin{proof}Multiplying equations \eqref{1.eq4} by $u$ and integrate over $x\in\Omega$ after the integration by parts, we get
\begin{equation}\label{1.est-dif1}
\frac12\frac d{dt}\|u\|^2_{L^2}+\|\nabla_x u\|^2_{L^2}=\frac12\((u_1^2,\partial_{x_2}v_2)+(u_2^2,\partial_{x_1}v_1)\)+(f(u),u)+(g,u).
\end{equation}
The first terms on the right-hand side of \eqref{1.est-dif1} can be estimated using the embedding $H^2\subset C$ and the maximal regularity of the Laplace operator:
\begin{equation*}
\frac12\(|(u_1^2,\partial_{x_2}v_2)|+|(u_2^2,\partial_{x_1}v_1)\)|\le\|u\|^2_{L^2}\|\nabla v\|_{L^\infty}\le C\|u\|^2_{L^2}\|\nabla u\|_{L^2}\le \frac14\|\nabla u\|^2+C\|u\|^2_{L^2}.
\end{equation*}
Analogously using that $f(u)$ has at most quadratic growth together with the interpolation inequality, we get
\begin{equation*}
|(f(u),u)|\le C(\|u\|^3_{L^3}+1)\le C(1+\|u\|_{L^2}^2\|\nabla u\|_{L^2})\le \frac14\|\nabla u\|^2_{L^2}+C(\|u\|^4_{L^2}+1).
\end{equation*}
Thus, \eqref{1.est-dif1} reads
\begin{equation}\label{1.est-dif2}
\frac d{dt}\|u\|^2_{L^2}+\|\nabla_x u\|^2_{L^2}\le C\|u\|^4_{L^2}+C(\|g\|^2_{L^2}+1).
\end{equation}
Inequality \eqref{1.est-dif2} together with the dissipative estimate \eqref{1.est-l1} are enough to verify the desired estimate \eqref{1.est-bur4} and finish the proof of the corollary. Indeed, for $t\in[0,1]$, we get \eqref{1.est-bur4} by applying the Gronwall inequality to \eqref{1.est-dif2} and using that the norm $\int_0^1\|u(s)\|^2_{L^2}\,ds$ is under the control. To obtain the estimate for $t\ge1$, we use the smoothing property
$$
\|u(1)\|^2_{L^2}\le Q(\|u(0)\|_{L^1})+Q(\|g\|_{L^2})
$$
which follows again from \eqref{1.est-dif2} by multiplying it by $t$ and applying the Gronwall inequality. From this estimate, we derive  the estimate
$$
\|u(t+1)\|^2_{L^2}\le Q(\|u(t)\|_{L^1})+Q(\|g\|_{L^2})
$$
which together with \eqref{1.est-l1} gives the desired dissipative estimate for $t\ge1$. Thus, the corollary is proved.
\end{proof}

\section{Fourth order equations with convective terms}\label{s2}
In this section, we apply the method of weighted estimates to some classes of fourth order parabolic equations with convective terms. We start with the following generalized Kuramoto-Sivashinsky equation in $\Omega=(-1,1)$:
\begin{equation}\label{2.eq1}
\partial_t u+\partial_x^4u+\lambda\partial_{x}^2u+\partial_x(u|u|^p)=f(u)+g,\ \ u\big|_{x=\pm1}=\partial_x^2 u\big|_{x=\pm1}=0,\ \ u\big|_{t=0}=u_0,
\end{equation}
where $\lambda\in \R$ and $p>0$ are some parameters, $g\in L^2(\Omega)$ is a given external force and $f\in C^1(\R,\R)$ is a given nonlinearity satisfying the growth restriction
\begin{equation}\label{2.f}
|f'(s)|\le C(1+|s|^q), \ \ \forall s\in \R,
\end{equation}
for some $q>0$. The next result gives the analogue of Theorem \ref{Th1.bur} for this equation.
\begin{theorem}\label{Th2.KS} Let the nonlinearity $f$ satisfy assumption \eqref{2.f} for some $q\le p$ and $g\in L^2(\Omega)$. Then, any sufficiently regular solution $u(t)$ of  problem \eqref{2.eq1} possesses the following estimate:
\begin{equation}\label{2.est-dis-KS}
\|u(t)\|^2_{L^2}+\int_t^{t+1}\left(\|\partial_{x}^2 u(s)\|^2_{L^2}+\|u(s)\|^{p+2}_{L^{p+2}}\right)\,ds\le C\|u_0\|_{L^2}^2e^{-\alpha t}+C(\|g\|^2_{L^2}+1),
\end{equation}
where the positive constants $\alpha$ and $C$ are independent of $t$, $u_0$ and $g$.
\end{theorem}
\begin{proof}We multiply equation \eqref{2.eq1} by $ue^{-Lx}$, integrate by parts and use that
\begin{equation}
(\partial_x(u|u|^p),ue^{-Lx})=\frac{p+1}{p+2}(\partial_x(|u|^{p+2}), e^{-Lx})=\frac{p+1}{p+2}L(|u|^{p+2},e^{-Lx}).
\end{equation}
Then, we get
\begin{multline}\label{2.equal}
\frac12\frac d{dt}\|u\|^2_{L^2_{e^{-Lx}}}+\|\partial_{x}^2u\|_{L^2_{e^{-Lx}}}^2+\frac{L^4+\lambda L^2}2\|u\|^2_{L^2_{e^{-Lx}}}+\\+\frac{p+1}{p+2}L(|u|^{p+2},e^{-Lx})=(L^2+\lambda)\|\partial_x u\|^2_{L^2_{e^{-Lx}}}+(f(u),ue^{-Lx})+(g,ue^{-Lx}).
\end{multline}
Using assumption \eqref{2.f} on the nonlinearity $f$ together with the fact that $p\ge q$, we may fix $L$ in such way that the nonlinear term on the right-hand side of \eqref{2.equal} can be estimated by the convective term on the left-hand side. Using also the interpolation inequality between $H^2$ and $L^2$ and the fact that $p>0$, we finally arrive at
\begin{equation}\label{2.36}
\frac d{dt}\|u\|^2_{L^2_{e^{-Lx}}}+\|u\|^2_{L^2_{e^{-Lx}}}+\alpha(\|\partial_x^2 u\|^2_{L^2}+\|u\|^{p+2}_{L^{p+2}})\le C_L(\|g\|^2_{L^2}+1)
\end{equation}
and the Gronwall inequality applied to this relation gives the desired estimate \eqref{2.est-dis-KS} and finishes the proof of the theorem.
\end{proof}
Note that, analogously to the previous section, estimate \eqref{2.est-dis-KS} is not strong enough to prevent blow up in finite time for the higher norms of the solution $u(t)$ for large exponents $p$. However, for sufficiently small values of $p$ it is sufficient to verify the global existence of  smooth solutions by the standard bootstrapping arguments. Since $H^1\subset L^\infty$ in the 1D case, it is sufficient to obtain the dissipative estimate in the $H^1$-norm which is done in the following corollary.
\begin{corollary}\label{Cor2.KS-h1} Let the assumptions of Theorem \ref{Th2.KS} hold and let, in addition $p\le6$. Then, problem \eqref{2.eq1} is globally well-posed in $H^1_0$ and the following estimate holds:
\begin{equation}\label{2.est-KS-h1}
\|u(t)\|_{H^1}\le Q(\|u_0\|_{H^1})e^{-\alpha t}+Q(\|g\|^2_{L^2}),
\end{equation}
where the positive constant $\alpha$ and a monotone function $Q$ are independent of $u$ and $t$.
\end{corollary}
\begin{proof} We multiply equation \eqref{2.eq1} by $\partial_x^2u$ and integrate over $x$. This gives
\begin{equation}\label{2.h1-est}
\frac12\frac d{dt}\|\partial_x u\|^2_{L^2}+\|\partial_x^3 u\|^2_{L^2}\le\lambda\|\partial_x u\|^2_{L^2}+(p+1)(|u|^p\partial_xu,\partial_x^2u)-(f(u),\partial_x^2u)-(g,\partial_x^2u).
\end{equation}
We assume for simplicity that $p\ge1$ and estimate below only the most complicated second term in the right-hand side of \eqref{2.h1-est} (the other terms are of lower order and are simpler to estimate). To this end, we integrate it by parts once more and use the interpolation inequalities
$$
\|u\|_{L^\infty}\le C\|u\|_{L^2}^{3/4}\|\partial_x^2u\|^{1/4}_{L^2},\ \ \|\partial_x u\|_{L^\infty}\le \|u\|_{L^2}^{1/4}\|\partial_{x}^2u\|_{L^2}^{3/4}
$$
and obtain
\begin{multline}
|(p+1)(|u|^p\partial_x u,\partial_x^2u)|\le C(|u|^{p-1},|\partial_x u|^3)\le C\|u\|_{L^\infty}^{p-1}\|\partial_x u\|_{L^\infty}\|\partial_x u\|^2_{L^2}\le\\\le C\|u\|_{L^2}^{\frac{3p-2}4}\|\partial_x^2 u\|_{L^2}^{\frac{p+2}4}\|\partial_x u\|^2_{L^2}.
\end{multline}
Thus, in order to control this term, we need the function $t\to\|u(t)\|_{L^2}^{\frac{3p-2}4}\|\partial_x^2 u(t)\|_{L^2}^{\frac{p+2}4}$ to be integrable. According to estimate \eqref{2.est-dis-KS} it will be indeed the case if $p\le6$ and the corollary is proved.
\end{proof}
Let us consider now the case where $q>p$. In this case, the solution $u(t)$ may blow up in finite time at least for some initial data $u_0$ and nonliearity $f(u)$. We demonstrate it on the following example:
\begin{equation}\label{2.eq2}
\partial_t u+\partial_x^4u+\lambda\partial_{x}^2u+\partial_x(u|u|^p)=|u|^{q+1}+g,\ \ u\big|_{x=\pm1}=\partial_x^2 u\big|_{x=\pm1}=0,\ \ u\big|_{t=0}=u_0.
\end{equation}
\begin{theorem}\label{Th2.blowup-KS} Let $q>p$, $g\in L^2(\Omega)$ and $q>0$. Then, there exists
smooth initial data $u_0$ such that the corresponding solution $u(t)$ of problem \eqref{2.eq2} blows up in finite time.
\end{theorem}
\begin{proof} The proof of this fact is based on the multiplication of equation \eqref{2.eq2} by $\varphi(x):=(\eb^2-x^2)^n$ and repeats word by word the proof of theorem \ref{Th1.blowup}. Indeed, the only new term is $(u,\varphi'''')$ which appear after the integration by parts in the term containing fourth derivative of $u$. This term can be estimated similar to \eqref{1.2d} if $n>\frac{4(q+1)}{q}$. So, leaving the details to the reader, we finish the proof of the theorem.
\end{proof}
We now switch to the so-called convective Cahn-Hilliard equation which gives another model example of fourth order 1D equations where the convective terms may prevent blow up. Namely, let us consider the following problem in $\Omega=(-1,1)$:
\begin{equation}\label{2.eq3}
\partial_{t} u+\partial_x^2\(\partial_x^2 u+f(u)\)+\partial_x(u|u|^p)=g,\ \ u\big|_{x=\pm1}=\partial_x^2u\big|_{x=\pm1}=0,\ \ u\big|_{t=0}=u_0,
\end{equation}
where the nonlinearity $f$ satisfies assumptions \eqref{2.f} and $g\in L^2(\Omega)$. In this case, the destabilizing nonlinearity $\partial_x^2f(u)$ is stronger than in the previous examples and we need stronger convective term in order to compensate it.
The next theorem is the analogue of Theorem \ref{Th2.KS} for this equation.
\begin{theorem}\label{Th2.KH} Let $g\in L^2(\Omega)$ and the nonlinearity $f$ satisfy \eqref{2.f}. Assume also that $p\ge2q$. Then, for any sufficiently smooth solution $u(t)$ of equation \eqref{2.eq3}, the following dissipative estimate holds:
\begin{equation}\label{2.dis-KH}
\|u(t)\|^2_{L^2}+\int_t^{t+1}\|\partial_{x}^2 u(s)\|^2_{L^2}+\|u(s)\|^{p+2}_{L^{p+2}}\,ds\le C\|u_0\|_{L^2}^2e^{-\alpha t}+C(\|g\|^2_{L^2}+1),
\end{equation}
where the positive constants $\alpha$ and $C$ are independent of $u_0$, $t$ and $g$.
\end{theorem}
\begin{proof} As before, we multiply equation \eqref{2.eq3} by $ue^{-Lx}$ and integrate over $x$. Then, analogously to \eqref{2.equal}, we get
\begin{multline}\label{2.equal-KH}
\frac12\frac d{dt}\|u\|^2_{L^2_{e^{-Lx}}}+\|\partial_{x}^2u\|_{L^2_{e^{-Lx}}}^2+
\frac{L^4}2\|u\|^2_{L^2_{e^{-Lx}}}+\frac{p+1}{p+2}L(|u|^{p+2},e^{-Lx})=L^2\|\partial_x u\|^2_{L^2_{e^{-Lx}}}+\\+(f(u),\partial_x^2u e^{-Lx})+L^2(f(u)u+\Phi(u),e^{-Lx})+(g,ue^{-Lx}),
\end{multline}
where $\Phi(u):=\int_0^uf(v)\,dv$. Using now assumption \eqref{2.f}, we see that
$$
|(f(u),\partial_x^2u e^{-Lx})|+|L^2(f(u)u+\Phi(u),e^{-Lx})|\le C\|u\|^{2q+2}_{L^{2q+2}_{e^{-Lx}}}+\frac12\|\partial_x^2 u\|^2_{L^2_{e^{-Lx}}}+CL^2\|u\|^{q+2}_{L^{q+2}_{e^{-Lx}}}.
$$
Since $p\ge2q$, we may fix $L$ large enough to absorb the right-hand side of the last estimate by the convective term in the left-hand side of \eqref{2.equal-KH} and end up (with the help of Young inequality) with estimate  \eqref{2.36}. This together with the Gronwall inequality give the desired estimate \eqref{2.dis-KH} and finish the proof of the theorem.
\end{proof}
As in the case of Kuramoto-Sivashinsky equation, estimate \eqref{2.dis-KH} is not strong enough to prevent blow up in higher norms if $p$ is large and some further restrictions on the exponents $p$ and $q$ are necessary for that. As in the case of Kuramoto-Sivashinsky equation, it is sufficient to verify the dissipative estimate in $H^1_0$-norm and the control of the higher norms will be obtained from it by the bootstrapping arguments. This estimate is given in the next corollary.
\begin{corollary}\label{Cor2.KH-dis-h1} Let $g\in L^2(\Omega)$, the nonlinearity $f$ satisfy \eqref{2.f} and, in addition, $p\le6$. Then, for any $u_0\in H^1_0(\Omega)$ the problem \eqref{2.eq3} is globally well-posed in $H^1$ and estimate \eqref{2.est-KS-h1} holds.
\end{corollary}
\begin{proof} Analogously to the proof of Corollary \ref{Cor2.KS-h1}, we multiply equation \eqref{2.eq3} by $\partial_x^2 u$ and integrate over $\Omega$. Then, we only need to estimate two terms containing nonlinearities. The convective nonlinearity $\partial_x(u|u|^p)$ is estimated in the proof of Corollary \ref{Cor2.KS-h1} and it is under the control if $p\le6$. The second nonlinearity $\partial_{x}^2f(u)$ can be estimated using assumption \eqref{2.f}, integration by parts and the interpolation inequlaity as follows
\begin{multline*}
|(\partial_x^2 f(u),\partial_x^2 u)|=|(f'(u)\partial_x u,\partial_x^3u)|\le C(\|u\|^{2q}_{L^\infty}+1)\|\partial_xu\|^2_{L^2}+\\+\frac12\|\partial_x^3u\|^2_{L^2}\le C(\|u\|_{L^2}^{3q/2}\|\partial_x^2u\|^{q/2}_{L^2}+1)\|\partial_xu\|^2_{L^2}+
\frac12\|\partial_x^3u\|^2_{L^2}.
\end{multline*}
Thus, keeping in mind the dissipative estimate \eqref{2.dis-KH}, we see that this term is also under the control if $q\le4$. This condition is automatically satisfied since due to the assumption $q\le p/2\le6$ and the corollary is proved.
\end{proof}
To conclude this section, we give the analogue of Theorem \ref{Th2.blowup-KS} for the case of the Cahn-Hilliard equation. As in the previous case, we consider the problem
\begin{equation}\label{2.eq4}
\partial_tu+\partial_x^2(\partial_x^2u+|u|^{q+1})+\partial_x(u|u|^p)=g,\ \ u\big|_{x=\pm1}=\partial_x^2u\big|_{x=\pm1}=0,\ \ u\big|_{t=0}=u_0.
\end{equation}
\begin{theorem}\label{Th2.blowup-KH} Let $g\in L^2(\Omega)$ and $p<q-1$. Then, there exist smooth initial data $u_0$ such that the corresponding solution $u(t)$ of equation \eqref{2.eq4} blows up in finite time.
\end{theorem}
\begin{proof} We multiply equation \eqref{2.eq4} by the function $\varphi(x):=(1-x^2)(x^4-14x^2+61)$ and integrate over $x\in(-1,1)$. Then, integrating by parts and using the boundary conditions, we have
\begin{equation}\label{2.blow}
\frac d{dt}(u,\varphi)+(|u|^{q+1},\varphi'')+(u,\varphi'''')=(u|u|^p,\varphi')+(g,\varphi).
\end{equation}
Using also  the obvious fact that
$$
\varphi''(x)=-30(1-x^2)(5-x^2),\ \ \varphi''''(x)=360(1-x^2),
$$
we see that the linear term in \eqref{2.blow} can be absorbed by the term containing $|u|^{q+1}$ and we end up with
\begin{equation}\label{2.blow1}
\frac d{dt}(u,\varphi)\ge\alpha(|u|^{q+1},\varphi)-C(|u|^{p+1},1)-C
\end{equation}
for some positive constants $\alpha$ and $C$. Thus, we only need to estimate the second term in the right-hand side of \eqref{2.blow1}. We will do this by H\"older inequality as follows:
$$
(|u|^{p+1},1)=(|u|^{p+1}\varphi^{\frac1{q+1}},\varphi^{-\frac1{q+1}})\le \(|u|^{q+1},\varphi\)^{\frac{p+1}{q+1}}\(\varphi^{-\frac{1}{q-p}},1\)^{\frac{q-p}{q+1}}.
$$
Since $p<q-1$, the last integral in the right-hand side is finite and applying the Young inequality, we finally arrive at
\begin{equation}
\frac d{dt}(u,\varphi)\ge \beta|(u,\varphi)|^{q+1}-C
\end{equation}
for some positive $\beta$ and $C$ which finishes the proof of the theorem.
\end{proof}
\begin{remark}\label{Rem2.bad} In contrast to the case of Kuramoto-Sivashinsky equation and reaction diffusion equations considered before, the obtained results are not complete. In particular, the behavior of solutions remain unclear if
\begin{equation}
\frac p2< q\le p+1.
\end{equation}
The similar equation with periodic boundary conditions has been studied in \cite{EdKaZe} using different methods. In particular, the examples of blowing up solutions for the case $q=2$ and $p=1$ were given there. The analogous method works for our case as well, but it gives the condition $q\ge 2p$ for the blow up which is weaker than our result if $p>1$.
\end{remark}

\section{KdV type equations}\label{s3}
In this section, we apply the above developed methods to study the generalized KdV equation equipped with Dirichlet boundary conditions:
\begin{equation}\label{3.eq1}
\partial_t u+\partial_x^3u=\partial_x(u|u|^p)+f(u)+g,\ \ u\big|_{x=-1}=u\big|_{x=1}=\partial_x u\big|_{x=1}=0,\ \ u\big|_{t=0}=u_0,
\end{equation}
where the nonlinearity $f$ satisfies assumptions \eqref{2.f}. We first note that, in contrast to the case of periodic boundary conditions, the KdV type equations with Dirichlet boundary conditions are not conservative and the corresponding linear equation possesses a smoothing property on a finite interval and even generates a $C^\infty$-semigroup, see \cite{KdV} for the details. By this reason, the local solvability properties of these equations are close to the parabolic ones, so we will concentrate only on the derivation of the proper a priori estimates. The next theorem is the analogue of Theorem \ref{Th1.bur} for this equation.
\begin{theorem} Let $g\in L^2(\Omega)$, $p>0$ and $p\ge q$. Then, any sufficiently regular solution of \eqref{3.eq1} possesses the following estimate:
\begin{equation}\label{3.dis-l2}
\|u(t)\|^2_{L^2}+\int_t^{t+1}\|\partial_x u(s)\|^2_{L^2}\,ds\le C\|u_0\|^2_{L^2}e^{-\alpha t}+C(\|g\|^2_{L^2}+1),
\end{equation}
where the positive constants $C$ and $\alpha$ are independent of $u_0$, $g$ and $t$.
\end{theorem}
\begin{proof} As before, we multiply the equation \eqref{3.eq1} by $ue^{Lx}$ and integrate over $\Omega$. Then, using the obvious integration by parts
\begin{multline*}
(\partial_x^3 u,ue^{Lx})=-(\partial_x^2 u,e^{Lx}\partial_x u)-L(\partial_x^2 u,ue^{Lx})=
\frac L2(|\partial_x u|^2,e^{Lx})+\\+L(|\partial_x^2 u|^2,e^{Lx})+L^2(\partial_x u,ue^{Lx})+e^{-L}|\partial_x u(-1)|^2=\\=\frac32 L(|\partial_x u|^2,e^{Lx})-\frac{L^3}{2}(|u|^2,e^{Lx})+e^{-L}|\partial_x u(-1)|^2
\end{multline*}
and using assumptions \eqref{2.f} on the nonlinearity $f$,
we end up with
\begin{multline}\label{1.est-dif}
\frac12\frac d{dt}\|u\|^2_{L^2_{e^{Lx}}}+\frac{3L}2\|\partial_x u\|^2_{L^2_{e^{Lx}}}+\frac{p+1}{p+2}L\|u\|^{p+2}_{L^{p+2}_{e^{Lx}}}\le\\\le C(\|u\|^{q+2}_{L^{q+2}_{e^{Lx}}}+1+\|g\|^2_{L^2_{e^{Lx}}})+C(L^3+1)\|u\|^2_{L^2_{e^{Lx}}}.
\end{multline}
Since $p>0$ and $p\ge q$, we may fix $L$ large enough that the terms containing $u$ on the right-hand side will be absorbed by the $L^{p+2}$-norm in the left-hand side. Applying then the Gronwall inequality to the obtained relation, we get the desired estimate \eqref{3.dis-l2} and finish the proof of the theorem.
\end{proof}
As before, the dissipative estimate \eqref{3.dis-l2} give the global well-posedness and smoothness of solutions of \eqref{3.eq1} if $p$ is small enough.
\begin{corollary}\label{Co3.smooth} Let $g\in L^2(\Omega)$, $0<p\le 2$ and $q\le p$. Then, problem \eqref{3.eq1} is globally well-posed in $L^2(\Omega)$. Moreover, $u(t)\in H^3$ for all $t>0$ and the following estimate holds:
\begin{equation}\label{3.h3-dis}
\|u(t)\|_{H^3}^2\le \frac{1+t^3}{t^3}\(Q(\|u_0\|_{L^2})e^{-\alpha t}+Q(\|g\|_{L^2})\),
\end{equation}
for some positive $\alpha$ and monotone function $Q$.
\end{corollary}
\begin{proof} We give below only the formal derivation of the corresponding estimates which can be justified in a straightforward way. We start with the uniqueness. Let $u_1(t)$ and $u_2(t)$ be two slutions of \eqref{3.eq1} and let $v=u_1-u_2$. Then, this function solves
\begin{equation}\label{3.dif}
\partial_t v+\partial_x^3 v=\partial_x(u_1|u_1|^p-u_2|u_2|^p)+f(u_1)-f(u_2).
\end{equation}
Multiplying this equation by $ve^{Lt}$ using assumptions \eqref{2.f} and arguing as before, we end up with
\begin{equation}\label{3.dif1}
\frac12\frac d{dt}\|v\|^2_{L^2_{e^{Lx}}}+\alpha \|\partial_x v\|^2_{L^2_{e^{Lx}}}+\alpha(|u_1|^p+|u_2|^p,v^2e^{Lx})\le C(\|u_1\|^{2p}_{L^\infty}+\|u_2\|^{2p}_{L^\infty}+1)\|v\|^2_{L^2_{e^{Lx}}}
\end{equation}
for some positive $C$ and $\alpha$ which are independent of $u_1$ and $u_2$. Using the interpolation inequality $\|u\|_{L^\infty}^2\le \|u\|_{L^2}\|\partial_x u\|_{L^2}$ together with the dissipative estimate \eqref{3.dis-l2} and assumption $p\le2$, we see that the $L^1$-norm in time of $\|u_i\|^{2p}_{L^\infty}$ is under the control and the Gronwall inequality applied to \eqref{3.dif1} gives
\begin{equation}\label{3.dif2}
\|v(t)\|^2_{L^2}\le C\|v(0)\|^2_{L^2}e^{Kt},
\end{equation}
where the constants $C$ and $K$ may depend on the $L^2$-norms of $u_i$. This proves the uniqueness.
\par
To verify the smoothing property \eqref{3.h3-dis}, we differentiate equation \eqref{3.eq1} in time and denote $w=\partial_t u$. Then this function solves
\begin{equation}\label{3.dif4}
\partial_t w+\partial_x^3 w=(p+1)\partial_x(|u|^pw)+f'(u)w.
\end{equation}
Multiplying this equation by $we^{Lx}$ and arguing analogously, we end up with
\begin{equation}\label{3.dif5}
\frac12\frac d{dt}\|w\|^2_{L^2_{e^{Lx}}}+\alpha \|\partial_x w\|^2_{L^2_{e^{Lx}}}+\alpha(|u|^p,w^2e^{Lx})\le C(\|u\|^{2p}_{L^\infty}+1)\|w\|^2_{L^2_{e^{Lx}}}.
\end{equation}
Applying the Gronvall inequality to this relation, we get
\begin{equation}\label{3.dif6}
\|\Dt u(t)\|^2_{L^2}\le Q(\|u_0\|^2_{L^2})e^{K(\|u_0\|_{L^2})t}\|\Dt u(0)\|^2_{L^2}
\end{equation}
for some monotone functions $Q$ and $K$. Using the obvious estimate
$$
\|\Dt u(0)\|_{L^2}\le Q(\|u_0\|_{H^3}),
$$
we see that the $L^2$-norm of $\partial_t u(t)$ is under the control. Let us derive the opposite estimate. To this end, we multiply equation \eqref{3.eq1} by $\partial_x^3u$ and use the H\"older inequality to obtain
\begin{equation}\label{3.l2h3}
\|\partial_x^3u(t)\|^2_{L^2}\le C\|\partial_t u(t)\|^2_{L^2}+C(\|u(t)\|^{4}_{L^\infty}(\|\partial_x u(t)\|^2_{L^2}+1)+\|g\|^2_{L^2}+1).
\end{equation}
Using the obvious interpolation, we get
$$
\|u\|^{4}_{L^\infty}\|\partial_x u\|^2_{L^2}\le C\|u\|^{10/3}_{L^2}\|\partial_x^3 u\|_{L^2}^{2/3}\|u\|_{L^2}\|\partial_x^3u\|_{L^2}=C\|u\|_{L^2}^{13/3}\|\partial_x^3u\|_{L^2}^{5/3}
$$
and, therefore,
\begin{equation}\label{3.dif7}
\|\partial_x^3 u(t)\|^2_{L^2}\le C(1+\|\partial_t u(t)\|_{L^2}^2+\|g\|^2_{L^2}+\|u(t)\|^{26}_{L^2}).
\end{equation}
Combining estimates \eqref{3.dif7} and \eqref{3.dif6}, we get the $H^3$-control of the solution $u(t)$:
\begin{equation}\label{3.dif8}
\|u(t)\|_{H^3}\le Q(\|u_0\|_{H^3}+\|g\|_{L^2})e^{K(\|u_0\|_{L^2})t}.
\end{equation}
However, this estimate is still not dissipative. To obtain the dissipative estimate, we need the smoothing property for $\partial_t u$. To this end, we note that
$$
\|\partial_t u(t)\|_{H^{-2}}^2\le C(\|\partial_x u(t)\|_{L^2}^2+\|u|u|^p\|_{H^{-1}}^2+\|f(u)\|_{H^{-2}}^2+\|g\|^2_{L^2})
$$
and
$$
|(u|u|^p,\varphi)|\le(\|u\|^{3}_{L^3}+1)\|\varphi\|_{L^\infty}\le (\|u\|^3_{L^3}+1)\|\varphi\|_{H^1_0}.
$$
Therefore,
$$
\|u|u|^p\|_{H^{-1}}^2\le C(\|u\|^6_{L^3}+1)\le C(\|u\|^4_{L^2}\|\partial_x u\|^2_{L^2}+1).
$$
Estimate for the lower order term $f(u)$ can be obtained analogously and we arrive at
\begin{equation}\label{3.eq9}
\|\Dt u(t)\|^2_{H^{-2}}\le C\((\|u\|^4_{L^2}+1)\|\partial_x u(t)\|_{L^2}^2+1+\|g\|^2_{L^2}\).
\end{equation}
We are now ready to finish the proof of the smoothing property. Indeed, we multiply \eqref{3.dif5} by $t$ and use the interpolation
\begin{equation}\label{3.dif-infty}
3t^{2}\|v(t)\|^2_{L^2_{e^{Lx}}}\le C\|v(t)\|_{H^{-2}}^{2/3}(t^{3/2}\|v(t)\|_{H^1})^{4/3}\le \frac\alpha2t^3\|\partial_x v(t)\|^2_{L^2_{e^{Lx}}}+C\|\partial_t u(t)\|^2_{H^{-2}}.
\end{equation}
Then, combining the last estimate with \eqref{3.eq9} we get
\begin{multline}\label{3.dif10}
\frac d{dt}(t^3\|v\|^2_{L^2_{e^{Lx}}})+\alpha t^3\|\partial_x v\|^2_{L^2_{e^{Lx}}}\le Ct^3(\|u\|^{2}_{L^2}\|\partial_x u\|^2_{L^2}+1)\|v\|^2_{L^2_{e^{Lx}}}+\\+C\((\|u\|^4_{L^2}+1)\|\partial_x u(t)\|_{L^2}^2+1+\|g\|^2_{L^2}\).
\end{multline}
Applying the Gronwall inequality to this relation and using the dissipative estimate \eqref{3.dis-l2}, we finally arrive at
\begin{equation}\label{3.dis11}
t^3\|\partial_t u(t)\|^2_{L^2}\le Q(\|u_0\|_{L^2})+Q(\|g\|_{L^2}),\ \ t\in(0,1]
\end{equation}
for some monotone function $Q$. In particular, fixing $t=1$ in this inequality, we have
\begin{equation}\label{3.dis12}
\|\partial_t u(t+1)\|_{L^2}\le Q(\|u(t)\|_{L^2})+Q(\|g\|_{L^2}).
\end{equation}
It only remains to note that estimates \eqref{3.dis11} and \eqref{3.dis12} together with \eqref{3.l2h3} and the dissipative estimate \eqref{3.dis-l2} give the desired estimate \eqref{3.h3-dis} and finish the proof of the corollary.
\end{proof}
Finally, we show that the blow up in finite time of smooth solutions of \eqref{3.eq1} becomes possible if the condition $p\ge q$ is violated. To this end, we consider the following equation:
\begin{equation}\label{3.eq2}
\partial_t u+\partial_x^3 u=\partial_x(u|u|^p)+|u|^{q+1}+g, \ \ u\big|_{x=-1}=u\big|_{x=1}=\partial_x u\big|_{x=1}=0,\ \ u\big|_{t=0}=u_0.
\end{equation}
Then, the following result holds.
\begin{theorem}\label{Th3.blow} Let $g\in L^2(\Omega)$, $q>0$ and $p<q$. Then, there exists somooth initial data $u_0$ such that the corresponding solution $u(t)$ of \eqref{3.eq2} blows up in finite time.
\end{theorem}
Indeed, the proof of this result is based on the multiplication of \eqref{3.eq2} by $\varphi(x):=(\eb^2-x^2)^n$ for sufficiently large $n$ and $\eb>0$ and repeats word by word the proof of Theorem \ref{Th2.blowup-KS}, so we leave the details to the reader.

\begin{remark}\label{Rem3.difference} We remind that the result of Corollary \ref{Co3.smooth} on the global existence of smooth solutions contains a restriction $p\le2$ (even in the case where $f(u)\equiv 0$) which looks unusual and too strict in comparison with  the classical theory of KdV equations on the whole line or on the circle. Indeed, in the defocusing case of a generalized KdV equation:
\begin{equation}\label{3.df}
\partial_t u+\partial_x^3 u=\partial_x(u|u|^p)
\end{equation}
which is considered in this section, the corresponding  Cauchy problem or periodic boundary value problem is globally well-posed in $H^1$ for any $p\ge0$, see \cite{BoSa,Tao} and references therein. This is related with the following energy conservation law for problem \eqref{3.df}:
\begin{equation}\label{3.conservation}
\partial_t\(\frac12|\partial_x u|^2+\frac1{p+2}|u|^{p+2}\)-\partial_x\(\partial_tu\partial_xu+\frac12(\partial_{x}^2u-u|u|^p)^2\)=0.
\end{equation}
Moreover, in the self-focusing case
\begin{equation}
\partial_tu+\partial_x^3u+\partial_x(u|u|^p)=0,
\end{equation}
where the energy is not sign-definite,
the critical exponent is $p=4$, namely, if $p<4$, the global existence of smooth solutions can be established and for $p=4$ the blow up of higher norms is occurred at least for the Caucgy problem with properly chosen initial data, see \cite{MaMe} and references therein.
\par
Unfortunately, identity \eqref{3.conservation} is not very helpful for the case of Dirichlet boundary conditions since its integration over $x$ leads to the extra term $\frac12|\partial_x^2u|^2\big|_{x=-1}^{x=1}$ which is out of control. Thus, we are unable to use the energy identity and this, in turn, leads to the restriction $p\le2$. Actually, we do not know whether or not this restriction is essential.
\end{remark}
\begin{remark}\label{Rem3.rev}
One more essential difference with the conservative case is the solvability of the generalized KdV backward in time. In the case of Cauchy problem or periodic boundary value problem, the corresponding problem is {\it reversible} in time due to the invariance with respect to the transformation $t\to-t$, $x\to-x$. Thus, the global solvability backward in time follows from the global solvability forward in time. That is clearly not the case for the Dirichlet boundary conditions already due to the smoothing property established in Corollary \ref{Co3.smooth}. Moreover, it can be shown that the only smooth solution of \eqref{3.df} with Dirichlet boundary conditions defined for all $t\in\R$ is $u\equiv0$. Indeed, the above mentioned transformation does not change the equation \eqref{3.df}, but changes the boundary conditions and leads to the problem
\begin{equation}\label{3.back}
\partial_tu+\partial_x^3u=\partial_x(u|u|^p),\ \ u\big|_{x=-1}=\partial_xu\big|_{x=-1}=u\big|_{x=1}=0.
\end{equation}
Multiplying this equation by $(1-x)u$ and integrating over $x$, we have
$$
\frac12\frac d{dt}\|u\|^2_{L^2_{1-x}}=\frac32\|\partial_xu\|^2_{L^2}+\frac{p+1}{p+2}\|u\|^{p+2}_{L^{p+2}}\ge C\|u\|_{L^2_{1-x}}^{p+2}
$$
which gives the desired result.
\end{remark}


\end{document}